\documentclass[preprint,12pt]{elsarticle}

\usepackage{amssymb}

\newtheorem{thm}{Theorem}

\newdefinition{dfn}{Definition}
\newdefinition{ex}{Example}
\newproof{proof}{Proof}
\newtheorem{cor}{Corollary}

\newtheorem{Lemma}{Lemma}

\begin{document}

\begin{frontmatter}



\title{Green's Function of a generalized boundary value transmission problem}


\author[rvt]{K. Aydemir}
\ead{kadriye.aydemir@gop.edu.tr} \cortext[cor1]{Corresponding Author
(Tel: +90 356 252 16 16, Fax: +90 356 252 15 85)}

\address[rvt]{Department of Mathematics, Faculty of Arts and Science, Gaziosmanpa\c{s}a University,\\
 60250 Tokat, Turkey}
 \journal{}

\begin{abstract}
In this study we give a comprehensive treatment for a new type
discontinuous BVP's  boundary conditions and transmission
(impulsive, jump or interface) conditions. A self-adjoint linear
operator  is defined in a suitable Hilbert space  such that the
eigenvalues of such a problem coincide with those of this operator.
 Then by suggesting an own approaches we construct Green's function for problem under
consideration and  showed that it has a compact resolvent for
corresponding inhomogeneous problem.
\end{abstract}


\end{frontmatter}


\section{Introduction}

 Boundary value problems can be investigate
through the methods of Green's function and eigenfunction expansion.
The main tool for solvability analysis of such problems is the
concept of Green's function. The concept of Green's functions is
very close to physical intuition (see \cite{du}). If one knows the
Green's function of a problem one can write down its solution in
closed form as linear combinations of integrals involving the
Green's function and the functions appearing in the inhomogeneities.
Green's functions can often be found in an explicit way, and in
these cases it is very efficient to solve the problem in this way.
Determination of Green's functions is also possible using
Sturm-Liouville theory. This leads to series representation of
Green's functions (see \cite{lev}). In this paper we shall
investigate a class  of BVP's which consist of the Sturm-Liouville
equation
\begin{equation}\label{1}
\mathcal{L}(u):=-\rho(x)u^{\prime \prime }(x)+ q(x)u(x)=\lambda
u(x), x\in\Omega
\end{equation}
together with eigenparameter-dependent boundary conditions at end
 points $x=a, b$
\begin{equation}\label{2}
\mathcal{L}_1(u):=\delta_1u(a)-\delta_2u'(a)-\lambda(\delta_3u(a)-\delta_4u'(a))=0,
\end{equation}
\begin{equation}\label{3}
\mathcal{L}_2(u):=\gamma_1u(b)-\gamma_2u'(b)+\lambda(\gamma_3u(b)-\gamma_4u'(b))=0,
\end{equation}
and  the transmission conditions at interior point $\xi_i \in (a,b),
i=1,2,...n$
\begin{equation}\label{4}
\pounds_i(u)=\delta^{-}_{i1}u'(\xi_i-)+\delta^{-}_{i0}u(\xi_i-)+\delta^{+}_{i1}u'(\xi_i+)+\delta^{+}_{i0}u(\xi_i+)=0,
\end{equation}
\begin{equation}\label{5}
\mathfrak{L}_i(u)=\gamma^{-}_{i1}u'(\xi_i-)+\gamma^{-}_{i0}u(\xi_i-)+\gamma^{+}_{i1}u'(\xi_i+)+\gamma^{+}_{i0}u(\xi_i+)=0,
\end{equation}
where  $\Omega=\bigcup\limits_{i=1}^{n+1}(\xi_{i-1}, \xi_{i}), \
a:=\xi_0, \ b:=\xi_{n+1}, \ \rho(x)=\rho_i^{2}>0  \ \textrm{for} \ x
\in \Omega_i:= (\xi_{i-1}, \xi_{i}), \ i=1,2,...n+1 $, the potential
$q(x)$ is real-valued function which continuous in each of the
intervals $(\xi_{i-1}, \xi_{i})$, and has a finite limits $q(
\xi_{i}\mp0)$, $\lambda$ \ is a complex spectral parameter, \
$\delta_{k}, \ \gamma_{k}\ (k=1,2,3,4), \ \delta^{\pm}_{ij}, \
\gamma^{\pm}_{ij} \ (i=1,2,...n \ \textrm{and} \ j=0,1)$ are real
numbers. We want emphasize that the boundary value problem studied
here differs from the standard boundary value problems in that it
contains transmission conditions and the eigenvalue-parameter
appears not only in the differential equation, but also in the
boundary conditions. Moreover the coefficient functions may have
discontinuity at one interior point. Naturally, eigenfunctions of
this problem may have discontinuity at the one inner point of the
considered interval. The problems with transmission conditions has
become an important area of research in recent years because of the
needs of modern technology, engineering and physics. Many of the
mathematical problems encountered in the study of
boundary-value-transmission problem cannot be treated with the usual
techniques within the standard framework of  boundary value problem
(see \cite{ji}). Note that some special cases of this problem arise
after an application of the method of separation of variables to a
varied assortment of physical problems. For example, some boundary
value problems with transmission conditions arise in heat and mass
transfer problems \cite{lik}, in vibrating string problems when the
string loaded additionally with point masses \cite{tik}, in
diffraction problems \cite{voito}. Also some problems with
transmission conditions which arise in mechanics (thermal conduction
problems for a thin laminated plate) were studied in  \cite{tite}.
\section{The fundamental solutions and characteristic function}
We shall define two solutions of the equation (\ref{1}) on the whole
$\Omega=\bigcup\limits_{i=1}^{n+1}(\xi_{i-1}, \xi_{i})$ by
$\phi(x,\lambda )=\phi_i(x,\lambda ) \ \textrm{for} \ x \in \Omega_i
\ \textrm{and} \ \chi(x,\lambda )=\chi_i(x,\lambda ) \ \textrm{for}
\  x \in \Omega_i (i=1,2,...,n+1)$
\begin{displaymath} \phi (x,\lambda
)=\left \{
\begin{array}{c}
\phi _{1}(x,\lambda ),  x\in \lbrack a,\xi_1) \\
\phi _{2}(x,\lambda ),  x\in (\xi_1,\xi_2)\\
\phi _{3}(x,\lambda ),  x\in (\xi_2,\xi_3) \\...\\
\phi _{n+1}(x,\lambda ),  x\in (\xi_n,b]\\
\end{array}\right.
\textrm{and} \ \chi(x,\lambda)=\left \{\begin{array}{ll}
\chi _{1}(x,\lambda ),  x\in \lbrack a,\xi_1) \\
\chi _{2}(x,\lambda ),  x\in (\xi_1,\xi_2)\\
\chi _{3}(x,\lambda ),  x\in (\xi_2,\xi_3) \\...\\
\chi _{n+1}(x,\lambda ),  x\in (\xi_n,b]\\
\end{array}\right.
\end{displaymath}
where$\phi_i(x,\lambda ) \ \ \textrm{and} \ \chi_i(x,\lambda )$ are
defined recurrently by following procedure. Let $\phi _{1}(x,\lambda
) \ \textrm{and} \ \chi _{n+1}(x,\lambda )$ be solutions of the
equation $(\ref{1})$ on $[a,\xi_1) \ \textrm{and } \ (\xi_n,b]$
satisfying initial conditions
\begin{eqnarray}&&\label{7}
u(a,\lambda)=\delta_2-\lambda \delta_4, \ u'(a,\lambda )=\delta_1 -\lambda \delta_3 \\
&&\label{10} u(b,\lambda)=\gamma_2+\lambda \gamma_4, \ u'(b,\lambda
)=\gamma_1 -\lambda \gamma_3
\end{eqnarray}
respectively. In terms of these solution we shall define the other
solutions $\phi _{i+1}(x,\lambda )\ \textrm{and}  \ \chi
_{i}(x,\lambda )$ by initial conditions
\begin{eqnarray}\label{8}
&&\phi _{i+1}(\xi_{i}+,\lambda)
=\frac{1}{\theta_{i12}}(\theta_{i23}\phi
_{i}(\xi_{i}-,\lambda)+\theta_{i24}\frac{\partial\phi
_{i}(\xi_{i}-,\lambda)}{\partial x})
\\ &&\label{9} \frac{\partial\phi _{i+1}(\xi_{i}+,\lambda)}{\partial x}
=\frac{-1}{\theta_{i12}}(\theta_{i13}\phi
_{i}(\xi_{i}-,\lambda)+\theta_{i14}\frac{\partial\phi
_{i}(\xi_{i}-,\lambda)}{\partial x})
\\ \nonumber
\textrm{and}
 \\ &&\label{11}
\chi _{i}(\xi_{i}-,\lambda)
=\frac{-1}{\theta_{i34}}(\theta_{i14}\chi
_{i+1}(\xi_{i}+,\lambda)+\theta_{i24}\frac{\partial\chi
_{i+1}(\xi_{i}+,\lambda)}{\partial x})
\\ && \label{12} \frac{\partial\chi _{i}(\xi_{i}-,\lambda)}{\partial x})
=\frac{1}{\theta_{i34}}(\theta_{i13}\chi
_{i+1}(\xi_{i}+,\lambda)+\theta_{i23}\frac{\partial\chi
_{i+1}(\xi_{i}+,\lambda)}{\partial x})
\end{eqnarray}
respectively, where $\theta_{ijk} \ (1\leq j< k \leq 4)$ denotes the
determinant of the j-th
 and k-th columns of the matrix
  $$
\left[\begin{array}{cccc}
 \delta^{+}_{i1} & \delta^{+}_{i0} & \delta^{-}_{i1} & \delta^{-}_{i0}   \\
 \gamma^{+}_{i1} & \gamma^{+}_{i0} & \gamma^{-}_{i1} & \gamma^{-}_{i0}
\end{array} %
 \right]  $$
for $i=1,2,...n.$ Everywhere in below we shall assume that
$\theta_{ijk}>0$ for all i,j,k.  The existence and uniqueness of
these solutions are follows from well-known theorem of ordinary
differential equation theory. Moreover by applying the method of
\cite{ka} we can prove that each of these solutions are entire
functions of parameter $\lambda \in \mathbb{C}$ for each fixed $x$.
Taking into account (\ref{8})-(\ref{12}) and the fact that the
Wronskians $ \omega_i(\lambda):=W[\phi_i(x,\lambda
),\chi_i(x,\lambda )]$ (i=1,2,...n+1) are independent of variable
$x$  we have
\begin{eqnarray*}
\omega_{i+1}(\lambda )&=&\phi_{i+1}(\xi_i+,\lambda
)\frac{\partial\chi_{i+1}(\xi_i+,\lambda )}{\partial x}-\frac{\partial\phi_{i+1}(\xi_i+,\lambda )}{\partial x}\chi_{i+1}(\xi_i+,\lambda ) \\
 &=&\frac{\theta_{i34}}{\theta_{i12}}(\phi_i(\xi_i-,\lambda )\frac{\partial\chi_i(\xi_i,\lambda )}{\partial x}
 -\frac{\partial\phi_i(\xi_i-,\lambda )}{\partial x}\chi_i(\xi_i-,\lambda )) \\
&=&\frac{\theta_{i34}}{\theta_{i12}} \omega_i(\lambda
)=\prod_{j=1}^{i}\frac{\theta_{j34}}{\theta_{j12}}\omega_1(\lambda )
\ (i=1,2,...n).
\end{eqnarray*}
It is convenient to define  the characteristic function  \
$\omega(\lambda)$ for our problem $(\ref{1})-(\ref{3})$ as
$$
\omega(\lambda):=\omega_1(\lambda)
=\prod_{j=1}^{i}\frac{\theta_{j12}}{
\theta_{i34}}\omega_{i+1}(\lambda ) \ (i=1,2,...n).
$$
Obviously, $\omega(\lambda)$ is an entire function. By applying the
technique of \cite{osh6} we can prove that there are infinitely many
eigenvalues $\lambda_{k}, \ k=1,2,...$ of the problem
$(\ref{1})-(\ref{5})$ which are coincide with the zeros of
characteristic function  \ $\omega(\lambda)$.
\section{\textbf{Operator treatment in a adequate Hilbert space }}
To analyze the spectrum of the BVTP $(\ref{1})-(\ref{5})$ we shall
construct an adequate  Hilbert space and define a symmetric linear
operator in it  such a way that the considered problem can be
interpreted as the eigenvalue problem of this operator. For this we
assume that $$ \ \kappa_{1}:= \left[%
\begin{array}{cccc}
  \delta_3 & \delta_4  \\
  \delta_1 & \delta_2
  \\
\end{array}%
 \right]>0,  \  \kappa_{2}:= \left[%
\begin{array}{cccc}
  \gamma_3 & \gamma_4  \\
  \gamma_1 & \gamma_2
  \\
\end{array} %
 \right]>0 $$
and introduce modified inner product on direct sum space
$\mathcal{H}_{1}= L_{2}(a,\xi_1)\oplus
L_{2}(\xi_1,\xi_2)\oplus...\oplus L_{2}(\xi_{n-1},\xi_n)\oplus
 L_{2}(\xi_n,b) \ \textrm{and} \
\mathcal{H}=\mathcal{H}_{1}\oplus \mathbb{C}^{2}$ by
\begin{eqnarray*}
&&<f,g>_{\mathcal{H}_{1}}:=\sum\limits_{k=0}^{n}\frac{1}{\rho_{k+1}^{2}}\prod\limits_{i=0}^{k}\theta_{i12}
\prod\limits_{i=k+1}^{n+1}\theta_{i34}\int\limits_{\xi_{k}+}^{\xi_{k+1}-}f(x)\overline{g(x)}dx
\end{eqnarray*}
 where $\theta_{012}=\theta_{(n+1)34}=1$ and
\begin{eqnarray*}
\label{ic} <F,G>_{\mathcal{H}}:=[f,g]_{\mathcal{H}_{1}} +
\prod\limits_{i=0}^{n}\theta_{i34}\frac{f_{1}\overline{g_{1}}}{\kappa_{1}}
+\prod\limits_{i=0}^{n}\theta_{i12}\frac{f_{2}\overline{g_{2}}}{\kappa_{2}}
\end{eqnarray*}
 for $F=\left(
  \begin{array}{c}
   f(x),
    f_{1}, f_{2} \\
  \end{array}
\right)$,\  $G=\left(
  \begin{array}{c}
    g(x),
  g_{1}, g_{2}  \\
  \end{array}
\right)$ $\in \mathcal{H}$ respectively. Obviously, these inner
products are equivalent to the standard inner products, so,
 $(\mathcal{H}, [.,.]_{\mathcal{H}})$
and $(\mathcal{H}_{1}, [.,.]_{\mathcal{H}_{1}})$ are also Hilbert
spaces. Let us now define the boundary functionals
\begin{eqnarray*}&&B_{a}[f]:= \delta_1f(a)-\delta_2f'(a),  \ \  B'_{a}[f]:=
\delta_3f(a)-\delta_4f'(a)\\
&&B_{b}[f]:=\gamma_1f(b)-\gamma_2f'(b),  \ \  B'_{b}[f]:=
\gamma_3f(b)-\gamma_4f'(b)
\end{eqnarray*}
and construct the operator $\mathfrak{R}:\mathcal{H}\rightarrow
\mathcal{H}$ with the domain
\begin{eqnarray*}\label{2.3}
dom(\mathfrak{R}):= &&\bigg \{F=(f(x), f_{1}, f_{2}):f(x), f'(x) \
\in \bigcap\limits_{i=1}^{i=n+1} AC_{loc}(\xi_{i-1}, \xi_{i}),\nonumber\\
&&\textrm{and has a finite limits} \ f(\xi_{i}\mp0) \ \textrm{and} \
f'(\xi_{i}\mp0), \ \mathcal{L }F \in
L_{2}[a,b],\nonumber\\&&\pounds_i(f)=\mathfrak{L}(f)=0,\ f_{1}=
B'_{a}[f], f_{2}=-B'_{b}[f]\bigg \}
\end{eqnarray*}
and action low$$ \mathcal{L}(f(x),B'_{a}[f], -B'_{b}[f])=(\ell f,
B_{a}[f], B_{b}[f]).$$ Then the problem $(\ref{1})-(\ref{5})$ can be
written in the operator equation form as
$$\mathfrak{R}F=\lambda F, \ \ F=(f(x), B'_{a}[f], -B'_{b}[f])
 \in dom(\mathfrak{R})$$
in the Hilbert space $\mathcal{H}$.
\begin{thm}\label{2.1}
The linear operator $\mathfrak{R}$  is symmetric.
\end{thm}
\begin{proof}
By applying the method of  \cite{osh6} it is not difficult to show
that $dom(\mathfrak{R})$ is dense in the Hilbert space
$\mathcal{H}$. Now let $F=(f(x),B'_{a}[f],
-B'_{b}[f]),G=(g(x),B'_{a}[g], -B'_{b}[g]) \in dom(\mathfrak{R}).$
By partial integration we have
\begin{eqnarray}\label{2.2}
&&<\mathfrak{R}F,G>_{\mathcal{H}}-<F,\mathfrak{R}G>_{\mathcal{H}} =
\theta_{134}\theta_{234}...\theta_{n34}(W(f, \overline{g};\xi_1-)-
W(f,
\overline{g};a)) \nonumber\\
&&+ \theta_{112}\theta_{234}...\theta_{n34}( W(f,
\overline{g};\xi_2-) -  W(f,
\overline{g};\xi_1+))\nonumber\\
&& +...
+\theta_{112}\theta_{212}...\theta_{n12}(W(f,\overline{g};b)-
W(f,\overline{g};\xi_n+))
\nonumber\\
&&+\frac{1}{\kappa_{1}}\prod\limits_{i=0}^{n}\theta_{i34}(B_{a}[f]\overline{B'_{a}[g]}-B'_{a}[f]\overline{B_{a}[g]})\nonumber\\
&&+
\frac{1}{\kappa_{2}}\prod\limits_{i=0}^{n}\theta_{i12}(B'_{b}[f]\overline{B_{b}[g]}-B_{b}[f]\overline{B'_{b}[g]})
\end{eqnarray}
where, as usual, $W(f, \overline{g};x)$ denotes the Wronskians of
the functions $f$ \textrm{and} $ \overline{g}$. From the definitions
of boundary functionals we get that
\begin{eqnarray}\label{cl} &&B_{a}[f]\overline{B'_{a}[g]}-B'_{a}[f]\overline{B_{a}[g]}=\kappa_{1}
W(f,\overline{g};a), \\&& \label{c2}
B'_{b}[f]\overline{B_{b}[g]}-B_{b}[f]\overline{B'_{b}[g]}=-\kappa_{2}
W(f,\overline{g};b)
\end{eqnarray}
Further, taking in view the definition of $\mathcal{L}$ and initial
conditions $(\ref{7})-(\ref{12})$ we derive that
\begin{eqnarray}\label{c3}\theta_{i34}W(f, \overline{g};\xi_i-)
=\theta_{i12} W(f, \overline{g};\xi_i+) \  i=1,2...n
\end{eqnarray}
Finally, substituting  (\ref{cl}),  (\ref{c2}) and (\ref{c3}) in
(\ref{2.2}) we have
\[<\mathfrak{R}F,G>_{\mathcal{H}}=<F,\mathfrak{R}G>_{\mathcal{H}} \  \textrm{for \ every} \  F,G \in dom(\mathfrak{R}), \] so the operator $\mathfrak{R}$ is
symmetric in $\mathcal{H}$. The proof is complete.
\end{proof}

\begin{thm}\label{2.1}
The operator $\mathfrak{R}$  is self-adjoint in $\mathcal{H}$ .
\end{thm}
\begin{cor}\label{rem2}
 If $f(x)$ \textrm{and} $g(x)$   are eigenfunctions
corresponding to distinct eigenvalues, then they are ,,orthogonal"
in the sense of
\begin{eqnarray}\label{2.3}
<f,g>_{\mathcal{H}_{1}}
+\prod\limits_{i=0}^{n}\theta_{i34}\frac{B'_{a}[f]B'_{a}[g]}{\kappa_{1}}
+\prod\limits_{i=0}^{n}\theta_{i12}\frac{B'_{b}[f]B'_{b}[g]}{\kappa_{2}}
=0
\end{eqnarray}
where $F=(f(x),B'_{a}[f], -B'_{b}[f]),G=(g(x),B'_{a}[g], -B'_{b}[g])
\in dom(\mathfrak{R})$.
\end{cor}
\begin{thm}\label{2.1}
The  operator $\mathfrak{R}$  has only point spectrum, i.e.
$\sigma(\mathfrak{R})=\sigma_p(\mathfrak{R})$.
\end{thm}
\begin{Lemma}\label{2.1}
The  operator $\mathfrak{R}$  has compact resolvent, i.e. for
$\forall \delta \in \mathbb{R}/\sigma_p(\mathfrak{R}),
(\mathfrak{R}-\delta I) $ is compact in $\mathcal{H}$.
\end{Lemma}
By applying the above results, we obtain the following theorem.
\begin{thm}\label{2.1}
The eigenfunctions of the problem $(\ref{1})-(\ref{5})$, augmented
to become eigenfunctions of \ $\mathfrak{R}$, are complete in
$\mathcal{H}$, i.e. if we let $\{\Psi_s= (\Psi_s(x),B'_{a}[\Psi_s],
-B'_b[\Psi_s]);n \in N\}$ be a maximum set of orthonormal
eigenfunctions of $\mathfrak{R}$, where $\{\Psi_s(x);s \in N\}$ are
eigenfunctions of \ the problem $(\ref{1})-(\ref{5})$, then for all
$(f(x),f_1,f_2) \in \mathcal{H}$,
\begin{eqnarray*}\label{2.3}f&=&\sum\limits_{s=1}^{\infty}\{\sum\limits_{k=0}^{n}\frac{1}{\rho_{k+1}^{2}}\prod\limits_{i=0}^{k}\theta_{i12}
\prod\limits_{i=k+1}^{n+1}\theta_{i34}\int\limits_{\xi_{k}+}^{\xi_{k+1}-}f(x)\Psi_s(x)dx+\prod\limits_{i=0}^{n}\theta_{i34}\frac{f_{1}B'_{a}[\Psi_s]}{\kappa_{1}}
\nonumber\\ &+&
\prod\limits_{i=0}^{n}\theta_{i12}\frac{f_{2}(-B'_b[\Psi_s])}{\kappa_{2}}\}\Psi_s.\end{eqnarray*}\label{2.3}
\end{thm}

\section{\textbf{Green's Function }}
 Now let $\lambda\in\mathbb{C}$ not be an eigenvalue of $\mathfrak{R}$ and
consider the operator equation
\begin{eqnarray}\label{141u}
(\lambda I-\mathfrak{R})u=f(x),
\end{eqnarray}
for arbitrary \ $u=(u(x),u_{1},u_{2}) \in \mathcal{H}$. This
operator equation is equivalent to the following inhomogeneous BVTP
\begin{eqnarray}\label{141}
&& \ \ \ \ \ \ \ \ \ \ \ \ \ \ \ (\lambda -\mathfrak{R} )u(x)= f(x),
\ \ x \in \Omega
\\ && \label{141v} \pounds_i(u)=\mathfrak{L}_i(u)=0, \ \
\mathcal{L}_1(u)=0, \ \ \mathcal{L}_2(u)=0
\end{eqnarray}
We shall search the general solution of the non-homogeneous
differential Equation (\ref{141}) in the form
\begin{equation}
u(x,\lambda )=\left\{
\begin{array}{c}
f_{11\lambda }(x)\phi_{1\lambda}(x)+f_{12\lambda}(x)\chi_{1\lambda}(x)%
\hbox{ for }x\in \left( a,\xi_i\right) \\
f_{21\lambda}(x)\phi_{2\lambda}(x)+f_{22\lambda}(x
)\chi_{2\lambda}(x) \hbox{ for }x\in ( \xi_1,\xi_2)
\\...\\ f_{i1\lambda}(x)\phi_{i\lambda}(x)+f_{i2\lambda}(x
)\chi_{i\lambda}(x) \hbox{ for }x\in (\xi_n,b)
\end{array}
\right.  \label{(6.6)}
\end{equation}
where the functions $f_{ij\lambda }(x):=f_{ij}(x,\lambda )
(i=1,2,...n+1,\ j=1,2)$ are the solutions of the system of equations
\begin{eqnarray*}
&&\left\{
\begin{array}{c}
f'_{11\lambda }(x)\phi_{1\lambda}(x)+f'_{12\lambda }\chi_{1\lambda}(x)=0 \\
f'_{11\lambda }\phi'_{1\lambda}(x)+f'_{12\lambda
}\chi'_{1\lambda}(x)=\frac{f(x)}{\rho_1^{2}}
\end{array},
\right. \left\{
\begin{array}{c}
f'_{21\lambda}(x)\phi_{2\lambda}(x)+f'_{22\lambda}(x)\chi_{2\lambda}(x)=0 \\
f'_{21\lambda}(x)\phi'_{2\lambda}(x)+f'_{22\lambda}(x)\chi'_{2\lambda}(x)=\frac{f(x)}{\rho_2^{2}}
\end{array}
\right.\\&& ......\left\{
\begin{array}{c}
f'_{i1\lambda }(x)\phi_{i\lambda}(x)+f'_{i2\lambda }\chi_{i\lambda}(x)=0 \\
f'_{i1\lambda }\phi'_{1\lambda}(x)+f'_{i2\lambda
}\chi'_{i\lambda}(x)=\frac{f(x)}{\rho_i^{2}}
\end{array}
\right.
\end{eqnarray*}
for $x\in ( a,\xi_1), \ x\in ( \xi_1,\xi_2),...  \ x \in (\xi_i,b) \
(i=1,2,...n+1)$, respectively. Since $\lambda $ is not an eigenvalue
$ \omega(\lambda) \neq 0 $. Then using the transmission conditions
we have
\begin{eqnarray*}
&&f_{11\lambda}(x)=\frac{1}{\rho_1^{2}\omega_1(\lambda
)}\int\limits_{x}^{\xi_1}u(y)\chi_{1\lambda}(y)dy+f_{11}(\lambda), \
x\in (a,\xi_1)\\
&&f_{12\lambda}(x)=\frac{1}{\rho_1^{2}\omega_1(\lambda
)}\int\limits_{a}^{x}u(y)\phi_{1\lambda}(y)dy+f_{12}(\lambda), \
x\in (a,\xi_1)\\
&&f_{21\lambda}(x)=\frac{1}{\rho_2^{2}\omega_2(\lambda
)}\int\limits_{x}^{\xi_2}u(y)\chi_{2\lambda}(y)dy+f_{21}(\lambda), \
x\in ( \xi_1,\xi_2)\\
&&f_{22\lambda}(x)=\frac{1}{\rho_2^{2}\omega_2(\lambda
)}\int\limits_{\xi_1}^{x}u(y)\phi_{2\lambda}(y)dy+f_{22}(\lambda), \
x\in ( \xi_1,\xi_2)\\&&...\\
&&f_{(n+1)1\lambda}(x)=\frac{1}{\rho_{n+1}^{2}\omega_{n+1}(\lambda
)}\int\limits_{x}^{b}u(y)\chi_{1\lambda}(y)dy+f_{(n+1)1}(\lambda), \
x\in (\xi_n,b)\\
&&f_{(n+1)2\lambda}(x)=\frac{1}{\rho_{n+1}^{2}\omega_{n+1}(\lambda
)}\int\limits_{\xi_n}^{x}u(y)\phi_{i\lambda}(y)dy+f_{(n+1)2}(\lambda),
\
x\in (\xi_n,b)\\
\end{eqnarray*}
where $f_{ij}(\lambda) \ (i,j=1,2)$  are arbitrary functions of
parameter $\lambda$. Substituting this into $(\ref{(6.6)})$ gives
\begin{eqnarray}\label{(6.13)}
u(x,\lambda)=\left\{\begin{array}{c}
               \frac{\chi_{1\lambda}(x)}{\rho_1^{2}\omega_1(\lambda
)}\int_{a}^{x}\phi_{1\lambda}(y)f(y)dy +
\frac{\phi_{1\lambda}(x)}{\rho_1^{2}\omega_1(\lambda
)}\int_{x}^{\xi_1}\chi_{1\lambda}(y)f(y)dy \\ +f_{11}(\lambda)\phi_{1\lambda}(x)+f_{12}(\lambda)\chi_{1\lambda}(x) \  \ \ \ \ \ \ \ \ for \  x \in (a,\xi_1) \\
 \frac{\chi_{2\lambda}(x)}{\rho_2^{2}\omega_2(\lambda
)}\int_{\xi_1}^{x}\phi_{2\lambda}(y)f(y)dy +
\frac{\phi_{2\lambda}(x)}{\rho_2^{2}\omega_2(\lambda
)}\int_{x}^{\xi_2}\chi_{2\lambda}(y)f(y)dy \\
+f_{21}(\lambda)\phi_{2\lambda}(x)+f_{22}(\lambda)\chi_{2\lambda}(x)
\  \ \ \ \ \ \ \ \ for \  x \in (\xi_1,\xi_2) \\...\\
\frac{\chi_{n+1\lambda}(x)}{\rho_{n+1}^{2}\omega_{n+1}(\lambda
)}\int_{\xi_n}^{x}\phi_{(n+1)\lambda}(y)f(y)dy +
\frac{\phi_{(n+1)\lambda}(x)}{\rho_{n+1}^{2}\omega_{n+1}(\lambda
)}\int_{x}^{b}\chi_{(n+1)\lambda}(y)f(y)dy \\
+f_{(n+1)1}(\lambda)\phi_{(n+1)\lambda}(x)+f_{(n+1)2}(\lambda)\chi_{(n+1)\lambda}(x)
\  \ \ \ \ \ \ \ \ for \  x \in (\xi_n,b)
\end{array}\right.
\end{eqnarray}
By differentiating we have
\begin{eqnarray}\label{(6.133)}
u'(x,\lambda)=\left\{\begin{array}{c}
               \frac{\chi'_{1\lambda}(x)}{\rho_1^{2}\omega_1(\lambda
)}\int_{a}^{x}\phi_{1\lambda}(y)f(y)dy +
\frac{\phi'_{1\lambda}(x)}{\rho_1^{2}\omega_1(\lambda
)}\int_{x}^{\xi_1}\chi_{1\lambda}(y)f(y)dy \\ +f_{11}(\lambda)\phi'_{1\lambda}(x)+f_{12}(\lambda)\chi'_{1\lambda}(x), \  \ \ \ \ \ \ \  \  x \in (a,\xi_1) \\
 \frac{\chi'_{2\lambda}(x)}{\rho_2^{2}\omega_2(\lambda
)}\int_{\xi_1}^{x}\phi_{2\lambda}(y)f(y)dy +
\frac{\phi'_{2\lambda}(x)}{\rho_2^{2}\omega_2(\lambda
)}\int_{x}^{\xi_2}\chi_{2\lambda}(y)f(y)dy \\
+f_{21}(\lambda)\phi'_{2\lambda}(x)+f_{22}(\lambda)\chi'_{2\lambda}(x),
 \ \ \ \ \ \ \  \  x \in (\xi_1,\xi_2) \\...\\
\frac{\chi'_{n+1\lambda}(x)}{\rho_{n+1}^{2}\omega'_{n+1}(\lambda
)}\int_{\xi_n}^{x}\phi_{(n+1)\lambda}(y)f(y)dy +
\frac{\phi_{(n+1)\lambda}(x)}{\rho_{n+1}^{2}\omega_{n+1}(\lambda
)}\int_{x}^{b}\chi_{(n+1)\lambda}(y)f(y)dy \\
+f_{(n+1)1}(\lambda)\phi'_{(n+1)\lambda}(x)+f_{(n+1)2}(\lambda)\chi'_{(n+1)\lambda}(x),
\  \ \ \ \ \ \ \ \  \  x \in (\xi_n,b)
\end{array}\right.
\end{eqnarray}
By using $(\ref{(6.13)}), (\ref{(6.133)})$ and the conditions
$(\ref{141v})$ we can derive that\\$f_{12}(\lambda)=0, \
f_{(n+1)}(\lambda)=0,$
\begin{eqnarray*}\label{19}f_{i1}(\lambda)&=&\sum\limits_{s=i+1}^{n+1}
\frac{1}{\rho_s^{2}\omega_s(\lambda)}\int_{\xi_{s-1}}^{\xi_s}\chi_{s\lambda}(y)f(y)dy,
\ \ i=1,2,...n
 \nonumber \\\label{21}
f_{(i+1)2}(\lambda)&=&\sum\limits_{s=1}^{i}
\frac{1}{\rho_s^{2}\omega_s(\lambda)}\int_{\xi_{s-1}}^{\xi_s}\phi_{s\lambda}(y)f(y)dy,
\ \ i=1,2,...n
\end{eqnarray*}
Putting  in (\ref{(6.13)}) gives
\begin{eqnarray}\label{22}
&u(x,\lambda)=\left\{\begin{array}{c}
\frac{\chi_{1\lambda}(x)}{\rho_1^{2}\omega_1(\lambda
)}\int_{a}^{x}\phi_{1\lambda}(y)f(y)dy +
\frac{\phi_{1\lambda}(x)}{\rho_1^{2}\omega_1(\lambda
)}\int_{x}^{\xi_1}\chi_{1\lambda}(y)f(y)dy \\
+\phi_{1\lambda}(x)\sum\limits_{s=2}^{n+1}
\frac{1}{\rho_s^{2}\omega_s(\lambda)}\int_{\xi_{s-1}}^{\xi_s}\chi_{s\lambda}(y)f(y)dy, \  \ \ \ \ \ \ \ \ \  x \in (a,\xi_1) \\
\frac{\chi_{2\lambda}(x)}{\rho_2^{2}\omega_2(\lambda
)}\int_{\xi_1}^{x}\phi_{2\lambda}(y)f(y)dy +
\frac{\phi_{2\lambda}(x)}{\rho_2^{2}\omega_2(\lambda
)}\int_{x}^{\xi_2}\chi_{2\lambda}(y)f(y)dy \\
+\phi_{2\lambda}(x) \sum\limits_{s=2}^{n+1}
\frac{1}{\rho_s^{2}\omega_s(\lambda)}\int_{\xi_{s-1}}^{\xi_s}\chi_{s\lambda}(y)f(y)dy\\
+\chi_{2\lambda}(x)\frac{1}{\rho_1^{1}\omega_1(\lambda)}\int_{\xi_{a}}^{\xi_1}\phi_{1\lambda}(y)f(y)dy,
\ \ \ \ \ \ \  x \in (\xi_1,\xi_2) \\...\\
\frac{\chi_{n+1\lambda}(x)}{\rho_{n+1}^{2}\omega_{n+1}(\lambda
)}\int_{\xi_n}^{x}\phi_{(n+1)\lambda}(y)f(y)dy +
\frac{\phi_{(n+1)\lambda}(x)}{\rho_{n+1}^{2}\omega_{n+1}(\lambda
)}\int_{x}^{b}\chi_{(n+1)\lambda}(y)f(y)dy \\+
\chi_{(n+1)\lambda}(x)\sum\limits_{s=1}^{n}
\frac{1}{\rho_s^{2}\omega_s(\lambda)}\int_{\xi_{s-1}}^{\xi_s}\phi_{s\lambda}(y)f(y)dy,
\  \ \ \ \ \ \ \ \  \  x \in (\xi_n,b)
\end{array}\right.
\end{eqnarray}
Let us introduce the Green's function as
\begin{eqnarray}\label{23}
G(x,y;\lambda)=\left\{\begin{array}{c}
\frac{\phi_{i\lambda}(y)\chi_{\lambda}(x)}{\omega_i(\lambda)}, \ \ \ a<y \leq x<b\, \ \ \ x,y\neq\xi_i \ Ý=1,2,...n \\

\frac{\phi_{\lambda}(x)\chi_{i\lambda}(y)}{\omega_i(\lambda)}, \ \ \ a<x \leq y<b\, \ \ \ x,y\neq\xi_i \ Ý=1,2,...n \\
\end{array}\right.
\end{eqnarray}
Then from (\ref{22}) and (\ref{23}) it follows that the considered
problem (\ref{141}), (\ref{141v}) has an unique solution given by
\begin{eqnarray}\label{2.16}
u(x,\lambda)&=& \int_{a}^{b} G(x,y;\lambda)f(y)dy
\end{eqnarray}

\end{document}